\documentclass[12pt]{nyjm} 

\usepackage{amssymb}
\usepackage{hyperref}
\usepackage[capitalize]{cleveref}

\renewcommand{\MR}[1]{\href{http://www.ams.org/mathscinet-getitem?mr=#1}{MR#1}}
\newcommand{\Zbl}[1]{\href{http://zbmath.org/?q=an:#1}{Zbl~#1}}

\title[Horospherical limit points of $S$-arithmetic groups]
{Horospherical limit points \\ of  $S$-arithmetic groups}

\author{Dave Witte Morris}
\address{Department of Mathematics and Computer Science, University of Lethbridge, Lethbridge, Alberta, T1K\,6R4, Canada}
\email{Dave.Morris@uleth.ca, http://people.uleth.ca/$\sim$dave.morris/}

\author{Kevin Wortman}
\address{Department of Mathematics, University of Utah, 
Salt Lake City, UT 84112-0090}
\email{wortman@math.utah.edu, http://www.math.utah.edu/$\sim$wortman/}

\date{\today}

\DeclareMathOperator{\CAT}{CAT}
\DeclareMathOperator{\Char}{char}

\newcommand{\real}{\mathbb{R}}
\newcommand{\rational}{\mathbb{Q}}
\newcommand{\integer}{\mathbb{Z}}
\newcommand{\closure}{\overline}

\renewcommand{\AA}{\mathbf{A}}
\newcommand{\GG}{\mathbf{G}}
\newcommand{\PP}{\mathbf{P}}

\newcommand{\RR}{\mathbf{R}} 
\newcommand{\TT}{\mathbf{T}} 
\newcommand{\EE}{\mathbf{E}} 
\newcommand{\MM}{\mathbf{M}} 
\newcommand{\NN}{\mathbf{N}} 
\newcommand{\weight}{\mathfrak{g}}

\newcommand{\bdry}{\partial}

\numberwithin{equation}{section}

\theoremstyle{plain}
\newtheorem{thm}[equation]{Theorem}
\newtheorem{prop}[equation]{Proposition}

\newtheorem{lem}[equation]{Lemma}

\crefname{cor}{Corollary}{Corollaries}
\Crefname{cor}{Corollary}{Corollaries}
\crefname{lem}{Lemma}{Lemmas}
\Crefname{lem}{Lemma}{Lemmas}

\theoremstyle{definition}
\newtheorem{notation}[equation]{Notation}
\newtheorem{defn}[equation]{Definition}
\newtheorem*{ack}{Acknowledgments}

\theoremstyle{definition}
\newtheorem{rem}[equation]{Remark}
\newtheorem{rems}[equation]{Remarks}

\crefformat{rems}{Remark~#2#1#3}

 \newcounter{case}
 
 \renewcommand{\thecase}{\arabic{case}}
 
 \renewcommand{\Cref}{\cref}
 \newcommand{\pref}[1]{(\ref{#1})}

\makeatletter
\newcommand{\noprelistbreak}{\smallskip\@nobreaktrue\nopagebreak} 
\makeatother

\makeatletter
\def\cref@thmoptarg[#1]#2#3#4{%
  \ifhmode\unskip\unskip\par\fi%
  \normalfont%
  \trivlist%
  \let\thmheadnl\relax%
  \let\thm@swap\@gobble%
  \thm@notefont{\fontseries\mddefault\upshape}%
  \thm@headpunct{.}
  \thm@headsep 5\p@ plus\p@ minus\p@\relax%
  \thm@space@setup%
  #2
  \@topsep \thm@preskip               
  \@topsepadd \thm@postskip           
  \def\@tempa{#3}\ifx\@empty\@tempa%
    \def\@tempa{\@oparg{\@begintheorem{#4}{}}[]}%
  \else%
    \refstepcounter[#1]{#3}
    \def\dth@counter{} 
    \def\@tempa{\@oparg{\@begintheorem{#4}{\csname the#3\endcsname}}[]}%
  \fi%
  \@tempa}
\makeatother

\begin{document}

\begin{abstract}
Suppose $\Gamma$ is an $S$-arithmetic subgroup of a connected, semisimple algebraic group~$\GG$ over a global field~$Q$ (of any characteristic). It is well known that $\Gamma$ acts by isometries on a certain $\CAT(0)$ metric space $X_S = \prod_{v \in S} X_v$, where each $X_v$ is either a Euclidean building or a Riemannian symmetric space. For a point~$\xi$ on the visual boundary of~$X_S$, we show there exists a horoball based at~$\xi$ that is disjoint from some $\Gamma$-orbit in~$X_S$ if and only if $\xi$~lies on the boundary of a certain type of flat in~$X_S$ that we call ``$Q$-good.''
This generalizes a theorem of G.\,Avramidi and D.\,W.\,Morris that characterizes the horospherical limit points for the action of an arithmetic group on its associated symmetric space~$X$.
\end{abstract}

\maketitle

\setcounter{tocdepth}{1} 
\tableofcontents

\section{Introduction}

\begin{defn}[{}{\cite[Defn.~B]{Hattori-GeomLimSets}}] \label{HoroLimDefn}
Suppose the group $\Gamma$ acts by isometries on the $\CAT(0)$ metric space~$X$, and fix $x \in X$.
A point~$\xi$ on the visual boundary of~$X$ is a \emph{horospherical limit point} for~$\Gamma$ if every horoball based at~$\xi$ intersects the orbit $x \cdot \Gamma$. Notice that this definition is independent of the choice of $x$. Also note that if $\Lambda$ is a finite-index subgroup of~$\Gamma$, then~$\xi$ is a horospherical limit point for~$\Lambda$ if and only if it is a horospherical limit point for~$\Gamma$. 
\end{defn}

In the situation where $\Gamma$ is an arithmetic group, with its natural action on its associated symmetric space~$X$, the horospherical limit points have a simple geometric characterization:

\begin{thm}[Avramidi-Morris {\cite[Thm.~1.3]{AvramidiMorris}}]
Let 
\noprelistbreak
	\begin{itemize} 
	\item $\GG$ be a connected, semisimple algebraic group over\/~$\rational$, 
	\item $K$ be a maximal compact subgroup of the Lie group\/ $\GG(\real)$,
	\item $X = K \backslash \GG(\real)$ be the corresponding symmetric space of noncompact type\/ \textup{(}with the natural metric induced by the Killing form of\/ $\GG(\real)$\textup{)},
	and
	\item $\Gamma$ be an arithmetic subgroup of\/~$\GG$.
	\end{itemize}
Then a point $\xi \in \bdry X$ is \textbf{not} a horospherical limit point for\/~$\Gamma$ if and only if $\xi$ is on the boundary of some flat~$F$ in~$X$, such that $F$ is the orbit of a\/ $\rational$-split torus in\/~$\GG(\real)$. 
 \end{thm}
 
This note proves a natural generalization that allows $\Gamma$ to be $S$-arithmetic (of any characteristic), rather than arithmetic. 
The precise statement assumes familiarity with the theory of Bruhat-Tits buildings \cite{Tits-RedGrpsLocFlds}, and requires some additional notation.

\begin{notation} \ 
\noprelistbreak
	\begin{enumerate}
	\item Let
\noprelistbreak
		\begin{itemize}
		\item $Q$ be a global field (of any characteristic),
		\item $\GG$ be a connected, semisimple algebraic group over~$Q$,
		\item $S$ be a finite set of places of~$Q$ (containing all the archimedean places if the characteristic of~$Q$ is~$0$),
		\item $G_v = \GG(Q_v)$ for each $v \in S$, where $Q_v$ is the completion of~$Q$ at~$v$,
		\item $K_v$ be a maximal compact subgroup of~$G_v$, for each $v \in S$,
		and
		\item $Z_S$ be the ring of $S$-integers in~$Q$.
		\end{itemize}
	\item Adding the subscript~$S$ to any symbol other than~$Z$ denotes the Cartesian product over all elements of~$S$. Thus, for example, we have $G_S =  \prod_{v \in S} G_v = \prod_{v \in S} \GG(Q_v)$.
	\item For each $v \in S$, let 
	$$ X_v = \begin{cases}
	\text{the symmetric space $K_v \backslash \GG(Q_v)$} & \text{if $v$~is archimedean} ,\\
	\text{the Bruhat-Tits building of $\GG(Q_v)$} & \text{if $v$~is nonarchimedean} 
	. \end{cases}
	$$
	Thus, $G_v = \GG(Q_v)$ acts properly and cocompactly by isometries on the $\mathrm{CAT}(0)$ metric space~$X_v$. So $G_S$ acts properly and cocompactly by isometries on the $\mathrm{CAT}(0)$ metric space~$X_S = \prod_{v \in S} X_v$.
	\end{enumerate}
\end{notation}

\begin{defn}
We say a family $\{Y_t\}_{t \in \real}$ of subsets of~$X_S$ is \emph{uniformly coarsely dense} in $X_S/\GG(Z_S)$ if there exists $C > 0$, such that, for every $t \in \real$, each $\GG(Z_S)$-orbit in~$X_S$ has a point that is at distance $< C$ from some point in~$Y_t$.
\end{defn}

See \cref{SarithQgoodDefn} for the definition of a $Q$-good flat in~$X_S$.

\begin{thm}[cf.\ {\cite[Cor.~4.5]{AvramidiMorris}}] \label{HoroLimSArith} 
For a point $\xi$ on the visual boundary of $X_S = \prod_{v \in S} X_v$, the following are equivalent:
\noprelistbreak
	\begin{enumerate}
	
	\item \label{HoroLimSArith-horo}
	$\xi$ is a horospherical limit point for\/~$\GG(Z_S)$.

	\item \label{HoroLimSArith-Qgood}
	$\xi$~is not on the boundary of any $Q$-good flat.

	\item \label{HoroLimSArith-parab}
	There does not exist a parabolic $Q$-subgroup\/~$\PP$ of\/~$\GG$, such that\/ $P_S$ fixes~$\xi$, and\/ $\PP(Z_S)$ fixes some\/ \textup(or, equivalently, every\/\textup) horosphere based at~$\xi$.

	\item \label{HoroLimSArith-SpheresDense}
	The horospheres based at~$\xi$ are uniformly coarsely dense in $X_S / \GG(Z_S)$.

	\item \label{HoroLimSArith-BallsDense}
	The horoballs based at~$\xi$ are uniformly coarsely dense in $X_S / \GG(Z_S)$.
	
	\item \label{HoroLimSArith-onto}
	$\pi(\mathcal{B}) = X_S/\GG(Z_S)$ for every horoball~$\mathcal{B}$ based at~$\xi$, where $\pi \colon X_S \to X_S/\GG(Z_S)$ is the natural covering map.

	\end{enumerate}
\end{thm}

\begin{rems} \ 
\noprelistbreak
	\begin{itemize}
	\item (\ref{HoroLimSArith-horo}~$\Leftrightarrow$~\ref{HoroLimSArith-onto}) is a restatement of \cref{HoroLimDefn}.
	\item (\ref{HoroLimSArith-SpheresDense}~$\Rightarrow$~\ref{HoroLimSArith-BallsDense}) is obvious,  because horoballs are bigger than horospheres.
	\item (\ref{HoroLimSArith-BallsDense}~$\Rightarrow$~\ref{HoroLimSArith-horo}) is well known (see, for example, 
		\cite[Lem.~2.3($\Leftarrow$)]{AvramidiMorris}). 
	\end{itemize}
The remaining implications \ref{HoroLimSArith-horo}~$\Rightarrow$~\ref{HoroLimSArith-Qgood}~$\Rightarrow$~\ref{HoroLimSArith-parab}~$\Rightarrow$~\ref{HoroLimSArith-SpheresDense} are proved in the following sections, by fairly straightforward adaptations of arguments in~\cite{AvramidiMorris}.

\end{rems}

\begin{ack}
D.\,W.\,M.~would like to thank A.\,Rapinchuk for answering his questions about tori over fields of positive characteristic.
K.\,W.~gratefully acknowledges the support of the National Science Foundation.
\end{ack}

\section{\texorpdfstring{Proof of (\ref{HoroLimSArith-parab}~$\Rightarrow$~\ref{HoroLimSArith-SpheresDense})}%
{Proof of (3 => 4)}}

(\ref{HoroLimSArith-parab}~$\Rightarrow$~\ref{HoroLimSArith-SpheresDense}) of \cref{HoroLimSArith} is the contrapositive of the following result.

\begin{prop}[cf.\ {\cite[Thm.~4.3]{AvramidiMorris}}] \label{SarithNotDense->parab} 
If the horospheres based at~$\xi$ are not uniformly coarsely dense in $X_S / \GG(Z_S)$, then there is a parabolic $Q$-subgroup\/~$\PP$ of\/~$\GG$, such that
\noprelistbreak
	\begin{enumerate}
	\item \label{SarithNotDense->parab-FixXi}
	$P_S$ fixes~$\xi$, 
	and
	\item \label{SarithNotDense->parab-FixHoro}
	$\PP(Z_S)$ fixes some\/ \textup(or, equivalently, every\/\textup) horosphere based at~$\xi$.
	\end{enumerate}
\end{prop}

\begin{proof} 
We modify the proof of \cite[Thm.~4.3]{AvramidiMorris}
to deal with minor issues, such as the fact that $G_S$ is not (quite) transitive on~$X_S$.
To avoid technical complications, assume $\GG$ is simply connected. We begin by introducing yet more notation:
\noprelistbreak
	\begin{itemize} \itemsep=\smallskipamount
	\item[($\Gamma$)] Let $\Gamma = \GG(Z_S)$.
	\item[($x$)] Let $x \in X_S$. If $v\in S$ is a nonarchimedean place, then we choose $x$ so that its projection to $X_v$ is a vertex.
	\item[($\gamma$)] Let $\gamma \colon \real \to X_S$ be a geodesic with $\gamma(0)=x$ and $\gamma(+\infty) = \xi$. Let $\gamma^+ \colon [0,\infty) \to X$ be the forward geodesic ray of~$\gamma$. For each $v \in S$, let $\gamma_v$ be the projection of~$\gamma$ to~$X_v$, so $\gamma_v$ is a geodesic in~$X_v$. 
 	\item[($F_S$)] For each $v \in S$, choose a maximal flat (or ``apartment'')~$F_v$ in~$X_v$ that contains~$\gamma_v$. Then $F_S$ is a maximal flat in~$X_S$ that contains~$\gamma$.
	\item[($A_S$)] For each $v \in S$, there is a maximal $Q_v$-split torus~$A_v$ of~$\GG(Q_v)$, such that $A_v$ acts properly and cocompactly on the Euclidean space~$F_v$ by translations. 
Then $A_S$ acts properly and cocompactly on~$F_S$ (by translations). 
	\item[($C_S$)] For each $v \in S$, choose a compact subset~$C_v$ of~$F_v$, such that $C_v A_v = F_v$. Then $C_S A_S = F_S$.
	\item[($A_\gamma$)] Let $A_\gamma = \{\, a \in A_S \mid C_S \, a \cap \gamma \neq \emptyset \,\}$ and $A_\gamma^+ = \{\, a \in A_S \mid C_S \, a \cap \gamma^+ \neq \emptyset \,\}$.
	\item[($F_\perp, A_\perp$)] Let $F_\perp$ be the (codimension-one) hyperplane in~$F_S$ that is orthogonal to the geodesic~$\gamma$ and contains~$x$. Let
		$$ A_\perp = \{\, a \in A_S \mid C_S \, a \cap F_\perp \neq \emptyset \,\} .$$

	\item[($P_v^\xi, N_v$)] For each $v \in S$, let 
		$$P_v^\xi = \bigl\{\, g \in \GG(Q_v) \mathrel{\big|} \text{$\{\, a g a^{-1} \mid a \in A_\gamma^+ \,\}$ is bounded} \,\bigr\} ,$$
so $P_v^\xi$ is a parabolic $Q_v$-subgroup of $\GG(Q_v)$ that fixes~$\xi$. The Iwasawa decomposition \cite[\S3.3.2]{Tits-RedGrpsLocFlds} allows us to choose a maximal horospherical subgroup~$N_v$ of~$\GG(Q_v)$ that is contained in~$P_v^\xi$ and is normalized by~$A_v$, such that $F_v \, N_v = X_v$.
	
	\item[\vbox to 10pt{\hbox{$\begin{pmatrix} P_v, M_v, \\ T_v, M^*_v \end{pmatrix}$}\vss}] By applying the $S$-arithmetic generalization of Ratner's Theorem that was proved independently by Margulis-Tomanov \cite{MargulisTomanov-InvtMeas} and Ratner \cite{Ratner-padic} (or, if $\Char Q \neq 0$,  by applying a theorem of Mohammadi \cite[Cor.~4.2]{Mohammadi-MeasHoroPosChar}), we obtain an $S$-arithmetic analogue of \cite[Cor.~2.13]{AvramidiMorris}. Namely, for some parabolic $Q$-subgroup~$\PP$ of~$\GG$,  if we let $P_v = \PP(Q_v)$ for each $v \in S$, and let $P_v = M_v T_v U_v$ be the Langlands decomposition over~$Q_v$ (so $T_v$ is the maximal $Q_v$-split torus in the center of the reductive group $M_v T_v$, and $U_v$ is the unipotent radical), then we have
	$$ \text{$N_S \subseteq M^*_S \, U_S$ \quad and \quad $M^*_S \, U_S \, \Gamma \subseteq \closure{N_S \, \Gamma}$} , $$
where $M^*_v$ is the product of all the isotropic almost-simple factors of~$M_v$. 
	\end{itemize}

Since $N_v \subseteq P_v$ for every~$v$ (and $P_S$ is parabolic), we have $U_S \subseteq N_S$ and $A_S \subset P_S$ (cf.\  proof of \cite[Lem.~2.10]{AvramidiMorris}). Therefore, since all maximal $Q_v$-split tori of~$P_v$ are conjugate \cite[Thm.~20.9(ii), p.~228]{Borel-LinAlgGrps}, and $M^*_v \, T_v$ contains a maximal $Q_v$-split torus, there is no harm in assuming $A_S \subseteq M^*_S \, T_S$, by replacing $M^*_S \,T_S$ with a conjugate.
Let $A_S^M = A_S \cap M_S = A_S \cap M^*_S$. 

Note that $N_v$ is in the kernel of every continuous homomorphism from $P_v^\xi$ to~$\real$. Since $P_v^\xi$ acts continuously on the set of horospheres based at~$\xi$, and these horospheres are parametrized by~$\real$, this implies that $N_v$ fixes every horosphere based at~$\xi$. Then, since $F_S \, N_S = X_S$, we see that, for each $a \in A_\gamma$, the set $F_\perp \, a \, N_S$ is the horosphere based at~$\xi$ through the point $xa$.  By the definition of $A_\perp$, this implies that the horosphere is at bounded Hausdorff distance from
	$$ \mathcal{H}_a = x a A_\perp \, N_S .$$
(Also note that every horosphere is at bounded Hausdorff distance from some~$\mathcal{H}_a$, since $A_S$ acts cocompactly on~$F_S$.)
We have 
	\begin{align} \label{HaClosure}
	\closure{ a  A_\perp \, N_S \, \Gamma }
	 \supseteq a A_\perp \cdot \closure{ N_S \, \Gamma}
	 \supseteq a A_\perp \cdot M^*_S \, U_S \, \Gamma
	. \end{align}

We claim that $F_\perp  A_S^M$ is not coarsely dense in~$F_S$. Indeed, suppose, for the sake of a contradiction, that the set is coarsely dense. Then $A_\perp A_S^M$ is coarsely dense in~$A_S$, which means there is a compact subset~$K_1$ of~$A_S$, such that $A_S = K_1  A_\perp A_S^M$. 
Also, the Iwasawa decomposition \cite[\S3.3.2]{Tits-RedGrpsLocFlds} of each $\GG(Q_v)$ implies there is a compact subset~$K_S$ of~$G_S$, such that $K_S A_S N_S = G_S$. Then, for every $a \in A_\gamma$, we have
	\begin{align*}
	K_S K_1  \cdot a A_\perp M_S^* U_S
	&= K_S a   (K_1  A_\perp M_S^*)  U_S
	\supseteq K_S a  A_S M_S^*  U_S 
	\supseteq K_S A_S N_S
	= G_S
	. \end{align*}
Since the compact set $K_S K_1$ is independent of~$a$, this (together with \pref{HaClosure}) implies that the sets $\mathcal{H}_a$ are uniformly coarsely dense in $X/\Gamma$. This contradicts the fact that the horospheres based at~$\xi$ are not uniformly coarsely dense.

Since $F_\perp$ is a hyperplane of codimension one in~$F_S$ (and $A_S^M$ is a group that acts by translations), the claim proved in the preceding paragraph implies $F_\perp = F_\perp  A_S^M \supseteq x A_S^M$. This means that $\gamma$ is orthogonal to the convex hull of $x A_S^M$.

On the other hand, we know that $M_S$ centralizes~$T_S$. Therefore, $M_S$ fixes the endpoint~$\xi_T$ of any geodesic ray~$\gamma_T$ in the convex hull of~$xT_S$. So $M_S$ acts (continuously) on the set of horospheres based at~$\xi_T$. However, $M_S$ is the almost-direct product of compact groups and semisimple groups over local fields, so it has no has no nontrivial homomorphism to~$\real$. (For the semisimple groups, this follows from the truth of the Kneser-Tits Conjecture \cite[Thm.~7.6]{PlatonovRapinchukBook}.) Since the horospheres are parametrized by~$\real$, we conclude that $M_S$ fixes every horosphere based at~$\xi_T$. Hence $A_S^M$ also fixes these horospheres. So $x A_S^M$ is contained in the horosphere through~$x$, which means the convex hull of $x A_S^M$ must be perpendicular to the convex hull of $x T_S$. Since $A_S^M T_S$ has finite index in~$A_S$, the conclusion of the preceding paragraph now implies that $\gamma$ is contained in the convex hull of $x T_S$, so $C_{G_S} \bigl( T_S \bigr)$ fixes~$\xi$. 

We also have
	\begin{align*}
	P_S = M_S T_S U_S = C_{G_S} \bigl( T_S \bigr) \, U_S \subseteq C_{G_S} \bigl( T_S \bigr) \, N_S .
	\end{align*}
Since $C_{G_S} \bigl( T_S \bigr)$ and~$N_S$ each fix the point~$\xi$, we conclude that $P_S$ fixes~$\xi$. 
This completes the proof of~\pref{SarithNotDense->parab-FixXi}. 

From here, the proof of~\pref{SarithNotDense->parab-FixHoro} is almost identical to the proof of \cite[Thm.~4.3(2)]{AvramidiMorris}.
\end{proof}

\section{\texorpdfstring{Proof of (\ref{HoroLimSArith-Qgood}~$\Rightarrow$~\ref{HoroLimSArith-parab})}{Proof of (2 => 3)}}

 (\ref{HoroLimSArith-Qgood}~$\Rightarrow$~\ref{HoroLimSArith-parab}) of \cref{HoroLimSArith} is the contrapositive of \cref{SarithParab->Qgood} below.

\begin{notation}
Suppose $\TT$ is a torus that is defined over~$Q$. Let
\noprelistbreak
	\begin{enumerate}
	\item $\mathcal{X}^*_Q(\TT)$ be the set of $Q$-characters of~$\TT$,
	and
	\item $T_S^{(1)} = \bigl\{\, g \in T_S \mathrel{\big|} \prod_{v \in S} \bigl\| \chi(g_v) \bigr\|_v = 1, \ \forall \chi \in \mathcal{X}_Q(\TT) \,\bigr\}$.
	\end{enumerate}
\end{notation}

\begin{defn} \label{SarithQgoodDefn} 
Suppose $\mathcal{F}$ is a flat in~$X_S$ (not necessarily maximal). We say $\mathcal{F}$ is \emph{$Q$-good} if there exists a $Q$-torus~$\TT$, such that 
\noprelistbreak
	\begin{itemize}
	\item $\TT$ contains a maximal $Q$-split torus of~$\GG$,  
	\item $\TT$ contains a maximal $Q_v$-split torus~$A_v$ of $G_v$ for every $v \in S$,
	\item $\mathcal{F}$ is contained in the maximal flat~$F_S$ that is fixed by~$A_S$,
	and
	\item $\mathcal{F}$ is orthogonal to the convex hull of an orbit of $T_S^{(1)}$ in~$F_S$.
	\end{itemize}
\end{defn}

\begin{rem}
 $Q$-good flats are a natural generalization of $\rational$-split flats. Indeed, the two notions coincide in the setting of arithmetic groups. Namely, suppose 
\noprelistbreak
	\begin{itemize}
	\item $Q$ is an algebraic number field, 
	\item $S$ is the set of all archimedean places of~$Q$, 
	\item $\TT$ is a maximal $Q$-split torus in~$\GG$,
	and
	\item $\mathbf{H} = \mathrm{Res}_{Q/\rational} \GG$ is the $\rational$-group obtained from~$\GG$ by restriction of scalars.
	\end{itemize}
Then $T_S$ can be viewed as the real points of a $\rational$-torus in~$\mathbf{H}(\mathbb{R})$, and $T_S^{(1)}$ is the group of real points of the $\rational$-anisotropic part of $T_S$. Thus, in this setting, the $Q$-good flats in the symmetric space of~${G}_S$ are naturally identified with the $\rational$-split flats in the symmetric space of $\mathbf{H}(\real)$. 
\end{rem}

\begin{prop}[cf.\ {\cite[Prop.~4.4]{AvramidiMorris}}] \label{SarithParab->Qgood}
If there is a parabolic $Q$-subgroup\/~$\PP$ of\/~$\GG$, such that 
	$P_S$ fixes~$\xi$, 
	and
	$\PP(Z_S)$ fixes every horosphere based at~$\xi$,
then $\xi$ is on the boundary of a $Q$-good flat in~$X_S$.
\end{prop}

\begin{proof} 
Choose a maximal $Q$-split torus~$\RR$ of~$\PP$. The centralizer of $\RR$ in $\GG$ is an almost direct product $\RR \MM$ for some reductive $Q$-subgroup $\MM$ of $\PP$.

Choose a $Q$-torus~$\mathbf{L}$ of~$\MM$, such that $\mathbf{L}(Q_v)$ contains a maximal $Q_v$-split torus~$B_v$ of~$\MM(Q_v)$ for each $v \in S$.  (This is possible when $\Char Q = 0$ by \cite[Cor.~3 of \S7.1, p.~405]{PlatonovRapinchukBook}, and the same proof works in positive characteristic, because a theorem of A.\,Grothendieck tell us that the variety of maximal tori is rational \cite[Exp.~XIV, Thm.~6.1, p.~334]{DemazureGrothendieck-SGA3},
\cite[Thm.~7.9]{BorelSpringer-Rationality2}.) Let $\TT=\RR\mathbf{L}$ and $A_v=\RR(Q_v)B_v$, so that $\TT$ is a $Q$-torus that contains the maximal $Q$-split torus $\RR$ as well as the maximal $Q_v$-split tori $A_v$ for all $v \in S$.

 Let $F_S$ be the maximal flat corresponding to~$A_S$, and choose some $x \in F_S$. Since $P_S$ fixes~$\xi$, there is a geodesic $\gamma = \{\gamma_t\}$ in~$F$, such that $\lim_{t \to \infty} \gamma_t = \xi$ (and $\gamma_0 = x$). 

Now $\TT(Z_S)$ is a cocompact lattice in $T_S^{(1)}$ (because the ``Tamagawa number'' of~$\TT$ is finite: see \cite[Thm.~5.6, p.~264]{PlatonovRapinchukBook} if $\Char Q = 0$; or see \cite[Thm.~IV.1.3]{Oesterle-Tamagawa} for the general case), and, by assumption, $\TT(Z_S)$ fixes the horosphere through~$x$. This implies that all of~$T_S^{(1)}$ fixes this horosphere, so $x \, T_S^{(1)}$ is contained in the horosphere. Therefore, the convex hull of $x \, T_S^{(1)}$ is perpendicular to the geodesic~$\gamma$, so $\gamma$ is a $Q$-good flat.
\end{proof}

\section{\texorpdfstring{Proof of (\ref{HoroLimSArith-horo}~$\Rightarrow$~\ref{HoroLimSArith-Qgood})}{Proof of (1 => 2)}}

(\ref{HoroLimSArith-horo}~$\Rightarrow$~\ref{HoroLimSArith-Qgood}) of \cref{HoroLimSArith} is the contrapositive of the following result.

\begin{prop}[cf.\ {\cite[Prop.~3.1]{AvramidiMorris} or \cite[Thm.~A]{Hattori-GeomLimSets}}] \label{SArithQGoodNotHoro}
If $\xi$ is on the boundary of a $Q$-good flat, then $\xi$~is not a horospherical limit point for $\GG(Z_S)$.
\end{prop}

\begin{proof}
Let:
\noprelistbreak
	\begin{itemize}
	\item $\mathcal{F}$ be a $Q$-good flat, such that $\xi$ is on the boundary of~$\mathcal{F}$.
	\item $\gamma$ be a geodesic in~$\mathcal{F}$, such that $\lim_{t \to \infty} \gamma(t) = \xi$.
	\item $\TT$, $A_S$, and~$F_S$ be as in \cref{SarithQgoodDefn}.
	\item $x = \gamma(0) \in F_S$.
	\item $F_S$ be considered as a real vector space with Euclidean inner product, by specifying that the point~$x$ is the zero vector.
	\item $C_x$ be a compact set, such that $C_x A_S = F_S$ (and $x \in C_x$).
	\item $\gamma^\perp$ be the orthogonal complement of the $1$-dimensional subspace~$\gamma$ in the vector space~$F_S$.
	\item $\gamma^\perp_A = \{\, a \in A_S \mid C_x a \cap \gamma^\perp \neq \emptyset \, \}$.
	\item $\gamma_A(t) \in A_S$, such that $\gamma(t) \in C_x \gamma_A(t)$, for each $t \in \real$.
	\item $\RR$ be a maximal $Q$-split torus of~$\GG$ that is contained in~$\TT$.
	\item $\Phi$ be the system of roots of~$\GG$ with respect to~$\RR$.
	\item $\alpha^S \colon T_S \to \real^+$ be defined by $\alpha^S(g) = \prod_{v \in S} \bigl\| \alpha(g_v) \bigr\|_v$ for $\alpha \in \Phi$ (where $\|\cdot\|_v \circ \alpha$ is extended to be defined on all of~$\TT(Q_v)$ by making it trivial on the $Q$-anisotropic part).
	\item $\hat\alpha^S \colon F_S \to \real$ be the linear map satisfying 
		$\hat\alpha^S (xa) = \log \alpha^S(a)$ for all $a \in A_S$.
	\item $\alpha^F \in F_S$, such that $\langle \alpha^F \mid y \rangle = \hat\alpha^S (y)$ for all $y \in F_S$.
	\item $\Phi^{++} = \{\, \alpha \in \Phi \mid \text{$\hat\alpha^S \bigl( \gamma(t) \bigr) > 0$ for $t > 0$} \,\}$.
	\item $\Delta$ be a base of~$\Phi$, such that $\Phi^+$ contains $\Phi^{++}$.
	\item $\Delta^{++} = \Delta \cap \Phi^{++}$.
	\item $\PP_\alpha = \RR_\alpha \MM_\alpha \NN_\alpha$ be the parabolic $Q$-subgroup corresponding to~$\alpha$, for $\alpha \in \Delta$, where
\noprelistbreak
	\begin{itemize}
		\item $\RR_\alpha$ is the one-dimensional subtorus of~$\RR$ on which all roots in $\Delta \smallsetminus \{\alpha\}$ are trivial,
		\item $\MM_\alpha$ is reductive with $Q$-anisotropic center,
		and
		\item the unipotent radical $\NN_\alpha$ is generated by the roots in~$\Phi^+$ that are \emph{not} trivial on~$\RR_\alpha$.
		\end{itemize}
	\end{itemize}

Given any large $t \in \real^+$, we know $\hat\alpha^S \bigl( \gamma(t) \bigr)$ is large for all $\alpha \in \Delta^{++}$. 
By definition, we have $T_S^{(1)} = \bigcap_{\alpha \in \Delta} \ker \alpha^S$. 
Since $\gamma$ is perpendicular to the convex hull of $x \cdot T_S^{(1)}$, this implies that $\gamma(t)$ is in the span of $\{ \alpha^F \}_{\alpha \in \Delta}$.
Also, for $\alpha \in \Delta$, we have
	$$ \langle \alpha^F \mid \gamma(t) \rangle = \hat\alpha^S \bigl( \gamma(t) \bigr) \ge 0 .$$
There is no harm in renormalizing the metric on~$X_S$ by a positive scalar on each irreducible factor (cf.\ \cite[Rem.~5.4]{AvramidiMorris}). This allows us to assume $\langle \alpha^F \mid \beta^F \rangle \le 0$ whenever $\alpha \neq \beta$ (see \cref{SArithInnerProdNeg} below).
Therefore, for any $b \in \gamma^\perp_A$, there is some $\alpha \in \Delta$, such that $\hat\alpha^S \bigl( x \gamma_A(t) b \bigr)$ is large (see \cref{InSpanMustLarge} below). This means $\alpha^S \bigl( \gamma_A(t) \, b \bigr)$ is large. 

Since conjugation by the inverse of $\gamma_A(t) \,  b$ contracts the Haar measure on $(N_\alpha)_S$ by a factor of $\alpha^S \bigl( \gamma_A(t) \,  b \bigr)^k$ for some $k \in \integer^+$, and the action of $N_S$ on $(N_\alpha)_S$ is volume-preserving, this implies that, for any $g \in \gamma_A(t) \,  b \, N_S$, conjugation by the inverse of~$g$ contracts the Haar measure on $(N_\alpha)_S$ by a large factor. Since $\NN_\alpha(Z_S)$ is a cocompact lattice in $(N_\alpha)_S$  (because the ``Tamagawa number'' of~$\NN_\alpha$ is finite: see \cite[Thm.~5.6, p.~264]{PlatonovRapinchukBook} if $\Char Q = 0$; or see \cite[Thm.~IV.1.3]{Oesterle-Tamagawa} for the general case), this implies there is some nontrivial $h \in \NN_\alpha(Z_S)$, such that 
	$\| g h g^{-1} - e \|$ is small.
We conclude that $\xi$ is not a horospherical limit point for $\GG(Z_S)$
 (cf.\ \cite[Lem.~2.5(2)]{AvramidiMorris}).
\end{proof}

\begin{lem} \label{SArithInnerProdNeg}
Assume the notation of the proof of \cref{SArithQGoodNotHoro}. The metric on~$X_S$ can be renormalized so that we have $\langle \alpha^F \mid \beta^F \rangle \le 0$ for all $\alpha,\beta \in \Delta$ with $\alpha \neq \beta$.
\end{lem}

\begin{proof}
When $v$ is archimedean, the Killing form provides a metric on~$X_v$. 
We now construct an analogous metric when $v$ is nonarchimedean. To do this, let $\Phi_v$ be the root system of~$\GG$ with respect to the maximal $Q_v$-split torus~$A_v$, let $\mathfrak{t} \oplus \bigoplus_{\alpha \in \Phi_v} \weight_\alpha$ be the corresponding weight-space decomposition of the Lie algebra of~$G_v$, choose a uniformizer~$\pi_v$ of~$Q_v$, let $\mathcal{X}_*(A_v)$ be the group of co-characters of~$A_v$, and define a $\integer$-bilinear form $\langle \ \mid \ \rangle_v \colon \mathcal{X}_*(A_v) \times \mathcal{X}_*(A_v) \to \real$ by
	$$ \langle \varphi_1 \mid \varphi_2 \rangle_v 
	= \sum_{\alpha \in \Phi_v} v \Bigl( \alpha \bigl( \varphi_1(\pi_v) \bigr) \Bigr) \, v \Bigl(  \alpha \bigl(\varphi_2(\pi_v) \bigr) \Bigr) 
	\bigl( \dim \weight_\alpha \bigr) .$$
This extends to a positive-definite inner product on $\mathcal{X}_*(A_v) \otimes \real$ (and the extension is also denoted by $\langle \ \mid \ \rangle_v$).
It is clear that this inner product is invariant under the Weyl group, so it determines a metric on~$X_v$ \cite[\S2.3]{Tits-RedGrpsLocFlds}. By renormalizing, we may assume that the given metric on~$X_v$ coincides with this one.

Let $\EE$ be the $Q$-anisotropic part of~$\TT$. Then it is not difficult to see that $\mathcal{X}_* \bigl( \RR \bigr) \otimes \real$ is the orthogonal complement of $\mathcal{X}_* \bigl( \EE(Q_v) \bigr) \otimes \real$, with respect to the inner product $\langle \ \mid \ \rangle_v$ (cf.\ \cite[Lem.~2.8]{AvramidiMorris}). 
Since every $Q$-root annihilates $\EE(Q_v)$, this implies that the $F_v$-component $\alpha^F_v$ of~$\alpha^F$ belongs to the convex hull of $x \, \RR(Q_v)$, for every $\alpha \in \Phi$.

From \cite[Cor.~5.5]{BorelTits-GrpRed}, we know that the Weyl group over~$Q$ is the restriction to~$\RR$ of a subgroup of the Weyl group over~$Q_v$. So the restriction of $\langle \ \mid \ \rangle_v$ to~$\mathcal{X}_* \bigl( \RR \bigr) \otimes \real$ is invariant under the $Q$-Weyl group. Assume, for simplicity, that $\GG$ is $Q$-simple, so the invariant inner product on $\mathcal{X}_* \bigl( \RR \bigr) \otimes \real$ is unique (up to a positive scalar).  (The general case is obtained by considering the simple factors individually.) This means that, after passing to the dual space $\mathcal{X}^* \bigl( \RR \bigr) \otimes \real$, the inner product $\langle \ \mid \ \rangle_v$ must be a positive scalar multiple~$c_v$ of the usual inner product (for which the reflections of the root system~$\Phi$ are isometries), so 
	$\langle \alpha^F_v \mid \beta^F_v \rangle_v = c_v \langle \alpha \mid \beta \rangle$
 for all $\alpha,\beta \in \Delta$. Since it is a basic property of bases in a root system that $\langle \alpha \mid \beta \rangle \le 0$ whenever $\alpha \neq \beta$, we therefore have
 	\begin{align*}
	\langle \alpha^F \mid \beta^F \rangle 
	= \sum_{v \in S} \langle \alpha^F_v \mid \beta^F_v \rangle_v
	= \sum_{v \in S} c_v \langle \alpha \mid \beta \rangle
	= \sum_{v \in S} \bigl( {> 0} \bigr)\bigl( {\le 0} \bigr)
	\le 0
	. & \qedhere \end{align*}
\end{proof}

\begin{lem}[{}{\cite[Lem.~2.6]{AvramidiMorris}}] \label{InSpanMustLarge}
Suppose
\noprelistbreak
	\begin{enumerate}
	\item $v,v_1,\ldots,v_n \in \real^k$, with $v \neq 0$,
	\item \label{InSpanMustLarge-span}
	$v$ is in the span of $\{v_1,\ldots,v_n\}$,
	\item \label{InSpanMustLarge-acute}
	$\langle v \mid v_i \rangle \ge 0$ for all~$i$,
	\item \label{InSpanMustLarge-obtuse}
	$\langle v_i \mid v_j \rangle \le 0$ for $i \neq j$,
	and
	\item $T \in \real^+$.
	\end{enumerate}
Then, for all sufficiently large $t \in \real^+$ and all $w \perp v$, there is some~$i$, such that $\langle tv + w | v_i \rangle > T$.
\end{lem}


\begin{thebibliography}{[99]}

\larger[2] 
\itemsep=\smallskipamount 

\bibitem{AvramidiMorris}
{\scshape Avramidi, G.; Morris, D.~W.}
Horospherical limit points of finite-volume locally symmetric spaces.
Preprint. 
\hfil\penalty100\hfilneg 
\url{http://arxiv.org/abs/1309.3554}


\bibitem{Borel-LinAlgGrps}
{\scshape Borel, A.}
Linear algebraic groups.
{\em Springer, New York,} 1991.
\MR{1102012} (92d:20001),
\Zbl{0726.20030}.

\bibitem{BorelSpringer-Rationality2}
{\scshape Borel, A.; Springer, T.~A.}
Rationality properties of linear algebraic groups~II.
{\em T\^ohoku Math. J}. (2) {\bf 20} (1968) 443--497. 
\MR{0244259} (39~\#5576),
\Zbl{0211.53302}

\bibitem{BorelTits-GrpRed}
{\scshape Borel, A.; Tits, J.}
Groupes r\'eductifs, 
{\em Inst. Hautes \'Etudes Sci. Publ. Math.} {\bf 27} (1965) 55--150.
\MR{207712} (34~\#7527),
\Zbl{0257.20032}




\bibitem{DemazureGrothendieck-SGA3}
{\scshape Demazure, M.; Grothendieck, A.}
Groupes de type multiplicatif, et structure des sch\'emas en groupes g\'en\'eraux
(SGA~3, tome~2).
Lect. Notes in Math., 152.
{\em Springer-Verlag, Berlin-New York,} 1970.
\MR{0274459} (43~\#223b),
\Zbl{0209.24201}
\hfil\penalty100\hfilneg 
\url{http://library.msri.org/books/sga/}

\bibitem{Hattori-GeomLimSets}
{\scshape Hattori, T.}
Geometric limit sets of higher rank lattices.
{\em Proc. London Math. Soc.} (3) {\bf 90} (2005) 689--710. 
\MR{2137827} (2006e:22013),
\Zbl{1077.22015}





\bibitem{MargulisTomanov-InvtMeas}
{\scshape Margulis, G.~A.; Tomanov, G.~M.}
Invariant measures for actions of unipotent groups over local fields on homogeneous spaces.
{\em Invent. Math.} {\bf 116} (1994) 347--392. 
\MR{1253197} (95k:22013),
\Zbl{0816.22004}

\bibitem{Mohammadi-MeasHoroPosChar}
{\scshape Mohammadi, A.}
Measures invariant under horospherical subgroups in positive characteristic.
{\em J. Mod. Dyn.} {\bf 5} (2011) 237--254.
\MR{2820561} (2012k:37010),
\Zbl{1258.37034}

\bibitem{Oesterle-Tamagawa}
{\scshape Oesterl\'e, J.}
Nombres de Tamagawa et groupes unipotents en caract\'eristique~$p$.
{\em Invent. Math.} {\bf 78} (1984) 13--88. 
\MR{0762353} (86i:11016),
\Zbl{0542.20024}

\bibitem{PlatonovRapinchukBook}
{\scshape Platonov, V.; Rapinchuk, A.}
Algebraic groups and number theory.
{\em Academic Press, New York,} 1994.
\MR{1278263} (95b:11039),
\Zbl{0841.20046}



\bibitem{Ratner-padic}
{\scshape Ratner, M.}
Raghunathan's conjectures for Cartesian products of real and p-adic Lie groups.
{\em Duke Math. J.} {\bf 77} (1995) 275--382. 
\MR{1321062} (96d:22015),
\Zbl{0914.22016}


\bibitem{Tits-RedGrpsLocFlds}
{\scshape Tits, J.}
Reductive groups over local fields.
{\em Automorphic forms, representations and $L$-functions}  
(Corvallis, Ore., 1977), Part~1, 29--69.
{\em Amer. Math. Soc., Providence, R.I.,} 1979.
 \MR{0546588} (80h:20064),
 \Zbl{0415.20035}



\end{thebibliography}
\end{document}